\documentclass[12pt]{amsart}
\usepackage{fullpage,colonequals}
\usepackage{amsmath,amsthm,amscd,amssymb,amsfonts}
\usepackage{mathtools,adforn,mathabx}
\usepackage{graphicx} 
\usepackage{xcolor} 
\usepackage[hidelinks]{hyperref}
\hypersetup{
  colorlinks,
  linkcolor={red!85!black},
  citecolor={blue!85!black},
  urlcolor={blue!95!black}
}
\usepackage[capitalize,nameinlink]{cleveref}

\renewcommand{\eqref}[1]{\hyperref[#1]{(\ref*{#1})}}
\crefformat{section}{#2\S#1#3}
\Crefformat{section}{#2\S#1#3}

\usepackage{mathrsfs,upgreek}
\usepackage[sans]{dsfont} 
\usepackage[math,condensed]{anttor}

\numberwithin{equation}{section} 
\numberwithin{figure}{section} 
\usepackage{enumitem} 
\setlist[enumerate]{label=$\arabic*.$, ref=$\arabic*$}

\theoremstyle{plain}
\newtheorem{theoalph}{Theorem}

\newtheorem{prop}{Proposition}[section]

\newtheorem{lemm}[prop]{Lemma}

\newtheorem{innercustomthm}{Main}
\newenvironment{custtheo}[1]
{\renewcommand\theinnercustomthm{#1}\innercustomthm}
{\endinnercustomthm}

\theoremstyle{definition}

\theoremstyle{remark}

\newtheoremstyle{citing}
{3pt}
{3pt}
{\itshape}
{}
{\bfseries}
{.}
{.5em}
{\thmnote{#3}}

\theoremstyle{citing}

\newcommand{\C}{\mathbb{C}}

\newcommand{\Q}{\mathbb{Q}}
\newcommand{\R}{\mathbb{R}}

\newcommand{\Z}{\mathbb{Z}}

\newcommand{\cK}{\mathcal{K}}

\newcommand{\fD}{\mathfrak{D}}

\newcommand{\fT}{\mathfrak{T}}

\newcommand{\sD}{\mathscr{D}}

\newcommand{\partn}[1]{{\smallskip \noindent \textbf{#1.}}}
\renewcommand{\:}{\colon}
\renewcommand{\=}{\colonequals}

\DeclareMathOperator{\diam}{diam}

\renewcommand{\emph}[1]{\textsf{\textit{#1}}}

\newcommand{\para}{a}
\newcommand{\hpara}{\widehat{a}}
\newcommand{\pspace}{\mathcal{A}}
\newcommand{\PK}{{\mathbb{P}^{1}_K}}
\newcommand{\PKber}{{\mathsf{P}^{1}_K}}

\newcommand{\aB}{\mathsf{B}}
\newcommand{\aD}{\mathsf{D}}
\newcommand{\aK}{\mathsf{K}}

\newcommand{\whr}{\widehat{r}}
\newcommand{\why}{\widehat{y}}

\renewcommand{\int}{\text{ *** \emph{CHANGE command} *** }}

\usepackage{microtype}
\begin{document}

\title[Irrational \textsc{Fatou} components]{Irrational \textsc{Fatou} components in non-\textsc{Archimedean} dynamics}
\author[J. Rivera-Letelier]{Juan Rivera-Letelier}
\address{Department of Mathematics, University of Rochester. Hylan Building, Rochester, NY~14627, U.S.A.}
\email{riveraletelier@gmail.com}
\urladdr{\url{http://rivera-letelier.org/}}
\date{\today}

\begin{abstract}
  This paper studies the geometry of \textsc{Fatou} components in non-\textsc{Archimedean} dynamics.
  By explicitly computing a wandering domain constructed by \textsc{Benedetto}, it provides the first example of a \textsc{Fatou} component that is an irrational disk.
\end{abstract}

\maketitle


\section{Introduction}
Throughout this paper, $K$ is an algebraically closed field that is complete with respect to a nontrivial ultrametric norm~$|\cdot|$.
Denote by~$K^{\times}$ its multiplicative subgroup and by~$|K^{\times}|$ its value group, which is defined by ${|K^{\times}| \= \{ |z| \: z \in K^{\times} \}}$.
For~$z_0$ in~$K$ and~$r$ in~$\R_{> 0}$, the disk
\begin{equation}
  \label{eq:1}
  \{ z \in K \: |z - z_0| < r \}
\end{equation}
is \emph{rational} if~$r$ belongs to~$|K^{\times}|$ and \emph{irrational} otherwise.
A \emph{rational} (resp. \emph{irrational}) \emph{disk} of the projective line~$\PK$ is the image of a rational (resp. irrational) disk by a \textsc{M{\"o}bius} transformation.

Let~$R$ be a rational map of degree at least two with coefficients in~$K$.
It defines a dynamical system acting on~$\PK$ and one acting on the \textsc{Berkovich} projective line~$\PKber$.
As in the complex setting, the \textsc{Fatou} set of~$R$ is the maximal domain of stability of~$R$ in~$\PKber$, and~$R$ maps each \textsc{Fatou} component onto a \textsc{Fatou} component.
For background, see~\cref{s:preliminaries} and, for example, \cite{Ben19b,Ber90,Jon15,Riv03c,0Riv0412}.

The geometry of periodic \textsc{Fatou} components is well understood: Each is either a rational disk, a finite intersection of rational disks, or of \textsc{Cantor} type \cite[\emph{Th{\'e}or{\`e}me de Classification}]{Riv03c}.
Moreover, every wandering \textsc{Fatou} component eventually maps to a disk \cite[Theorem~11.2]{Ben19b}, and several constructions yield a rational one.
For example, every \emph{trivial} wandering \textsc{Fatou} component---one containing a point asymptotic to a periodic orbit in~$\PKber$---eventually maps to a rational disk.
In addition, every \textsc{Latt{\`e}s} map with persistent bad reduction admits a nontrivial wandering \textsc{Fatou} component that is a rational disk if the $\Q$\nobreakdash-rank of~$|K^{\times}|$ is at least~$2$.

The existing constructions of wandering \textsc{Fatou} components leave open the question of whether there is one that is an irrational disk \cite{Ben02a,Ben06,0Fer0503,0Nop2404,Tru15}.
The following result answers this question affirmatively.
Recall that a multiplicative subgroup~$\sD$ of~$\R_{>0}$ is \emph{divisible}, if, for all~$r$ in~$\R_{> 0}$ and~$n$ in~$\Z_{> 0}$, there is~$\whr$ in~$\sD$ satisfying~${\whr^n = r}$.
The value group of every algebraically closed ultrametric field is a divisible subgroup of~$\R_{> 0}$.

\begin{theoalph}[Irrational \textsc{Fatou} Components]
  \label{t:wandering-irrationally}
  Suppose that the residue characteristic of~$K$ is strictly positive and~$|K^{\times}|$ is a proper subgroup of~$\R_{> 0}$.
  Moreover, let~$\sD$ be a proper divisible subgroup of~$\R_{> 0}$ containing~$|K^{\times}|$.
  Then, there is a polynomial with coefficients in~$K$ with a \textsc{Fatou} component whose diameter is outside~$\sD$.
  In particular, there is a \textsc{Fatou} component that is an irrational disk.
\end{theoalph}

Existing examples of wandering \textsc{Fatou} components that are rational disks are either trivial or occur for rational maps that are not polynomials.
The following result is that a polynomial can admit a nontrivial one.

\begin{theoalph}[Rational Wandering Domains]
  \label{t:wandering-rationally}
  Suppose that the residue characteristic of~$K$ is strictly positive.
  Then, there is a polynomial with coefficients in~$K$ with a nontrivial wandering \textsc{Fatou} component that is a rational disk.
\end{theoalph}

\textsc{Baker} \cite{Bak76} provided the first examples of complex entire functions with a (multiply-connected) wandering \textsc{Fatou} component.
See~\cite{Ber11,KisShi08} for further examples.
\textsc{Sullivan} showed that every \textsc{Fatou} component of a complex rational map is preperiodic~\cite{Sul85a}.
\mbox{\textsc{Astorg}}, \textsc{Buff}, \textsc{Dujardin}, \textsc{Peters}, and \textsc{Raissy} \cite{AstBufDujPetRai16} gave the first example of a polynomial endomorphism of~$\C^2$ with a wandering \textsc{Fatou} component, and \textsc{Berger} and \textsc{Biebler} \cite{BerBie23} later provided the first example of a polynomial automorphism of~$\C^2$ with the same property.

\textsc{Benedetto}~\cite{Ben02a,Ben06} gave the first examples of a non-\textsc{Archimedean} rational map with a nontrivial wandering \textsc{Fatou} component.
These examples occur when the residue characteristic~$p$ of~$K$ is strictly positive and involve polynomials of degree~$p + 1$.
See also the examples provided by \mbox{\textsc{Nopal-Coello}} \cite{0Nop2404} involving quadratic rational maps with ${p = 2}$.
\textsc{Fern{\'a}ndez Lamilla} \cite{0Fer0503} showed that such examples are, in fact, common: The set of rational maps of degree~$d$ at least~$p + 1$ with a nontrivial wandering \textsc{Fatou} component is somewhere dense in the space of rational maps of degree~$d$.
In these results, wild ramification plays a crucial r{\^o}le.
\textsc{Trucco} \cite{Tru15} provided examples of polynomials with a nontrivial wandering \textsc{Fatou} component without wild ramification.
These wandering components arise from the \textsc{Fibonacci} combinatorics of the map.

The main ingredient in the proof of Theorems~\ref{t:wandering-irrationally} and~\ref{t:wandering-rationally} is an explicit computation of specific wandering \textsc{Fatou} components (\cref{t:main} in~\cref{s:main-theorem}) constructed by \textsc{Benedetto} in~\cite{Ben02a}.
To construct nontrivial wandering \textsc{Fatou} components, the papers~\cite{Ben02a,Ben06,0Fer0503,0Nop2404} compute diameters of the iterates of a small ball that remain outside the wild ramification locus, where the map acts locally as an affine map.
Computing the diameter of a nontrivial wandering \textsc{Fatou} component exactly requires controlling the diameter of the iterates of a ball that passes through the wild ramification locus infinitely often.
The computations are more delicate.

\subsection{Strategy and organization}
\label{ss:organization}
\cref{s:main-theorem} derives Theorems~\ref{t:wandering-irrationally} and~\ref{t:wandering-rationally} from a result that explicitly computes the diameter of certain nontrivial wandering \textsc{Fatou} components (\cref{t:main}).
The itinerary choice leads to cancellations that make the computation of the exact diameter tractable.
An increasing sequence of integers~$(\ell_s)_{s = 0}^{+\infty}$, indexing the times the wandering component passes through the wild ramification locus, encodes the itinerary.
The formula for the diameter depends solely on~$(\ell_s)_{s = 0}^{+\infty}$.
For example, ${(\ell_s)_{s = 0}^{+\infty} = (s)_{s = 0}^{+\infty}}$ yields a nontrivial wandering \textsc{Fatou} component that is a rational disk, thus proving \cref{t:wandering-rationally}.
The construction of an irrational disk that is a nontrivial wandering \textsc{Fatou} component and the proof of \cref{t:wandering-irrationally}, proceeds in two steps.
The first step is to describe a concrete \textsc{Cantor} set of possible diameter values for wandering components provided by the \cref{t:main} (\cref{l:wandering-diadically}).
The second step is to prove that not all elements of this \textsc{Cantor} set can belong to the same proper divisible subgroup of~$\R_{> 0}$.

After some preliminary considerations in~\cref{s:preliminaries}, the proof of the \cref{t:main} proceeds in two independent steps.
\cref{s:itinerary-realization} provides the first step, which is a sufficient condition for an itinerary to be realized by a point and a parameter (\cref{p:itinerary-realization}).
The proof is a streamlined version of the arguments in~\cite{Ben02a}.
The second, and main, step explicitly computes the diameters of the iterates of a specific ball as it passes through the wild ramification locus (\cref{p:explicit-diameter}).
The proof occupies \cref{ss:proof-explicit-diameter} and that of the \cref{t:main}~\cref{ss:proof-main-theorem}.

\subsection{Acknowledgments}
I thank V{\'{\i}}ctor \textsc{Nopal-Coello} and the referees for their comments, which helped improve the exposition.
I am especially grateful to one of the referees for providing the reference \cite[Theorem~11.2]{Ben19b} and for reading the paper in detail.

I wrote most of this article while visiting \textsc{Brown} University and participating in the ``Complex and Arithmetic Dynamics'' program at the Institute for Computational and Experimental Research in Mathematics (ICERM).
I thank these institutions for the optimal working conditions provided, and acknowledge partial support from FONDECYT grant 1100922.

\section{Main Theorem}
\label{s:main-theorem}
This section derives Theorems~\ref{t:wandering-irrationally} and~\ref{t:wandering-rationally} from the \cref{t:main}, stated below.
It computes, within the family of polynomials introduced by \textsc{Benedetto} in~\cite{Ben02a}, the diameter of a wandering component of the filled \textsc{Julia} set whose itinerary has a specific form.
The rest of the paper is devoted to proving the \cref{t:main}.

Throughout the rest of this paper, suppose that the residue characteristic~$p$ of~$K$ is a prime number, and put
\begin{equation}
  \label{eq:2}
  q
  \=
  (p - 1)(2p^2 - 2p - 1)
  \text{ and }
  \kappa
  \=
  \frac{p^{2p - 2}}{p^{2p - 1} - p + 1}.
\end{equation}

Put
\begin{equation}
  \label{eq:3}
  \pspace
  \=
  \left\{ \para \in K \: |\para| > 1 \text{ and } \left| ap^{p - 1} \right| \le 1 \right\}
\end{equation}
and for every~$\para$ in~$\pspace$ define the polynomial~$P_\para(z)$ with coefficients in~$K$ by
\begin{equation}
  \label{eq:4}
  P_\para(z)
  \=
  \para z^p + \left( 1 - \para \right) z^{p + 1}.
\end{equation}
The \emph{filled \textsc{Julia} set~$\cK(P_\para)$ of~$P_\para$} is defined by
\begin{equation}
  \label{eq:5}
  \cK(P_\para)
  \=
  \left\{ z \in K \: ( P_\para^n(z) )_{n = 1}^{+\infty} \text{ is bounded} \right\}.
\end{equation}
Clearly, ${P_\para(\cK(P_\para)) \subseteq \cK(P_\para)}$.
Given~$z_0$ in~$K$ and~$r$ in~$\R_{> 0}$, the ball~$\{ z \in K \: |z - z_0| \le r \}$ is a \emph{component of~$\cK(P_\para)$} if it is contained in~$\cK(P_\para)$ and if for every real number~$r'$ satisfying ${r' > r}$, the disk ${\{ z \in K \: |z - z_0| < r' \}}$ is not contained in~$\cK(P_\para)$.
The polynomial~$P_\para$ maps every component of~$\cK(P_\para)$ onto a component of~$\cK(P_\para)$.
A component~$B$ of~$\cK(P_\para)$ is \emph{wandering} if, for all distinct~$i$ and~$j$ in~$\Z_{\ge 0}$, the sets~$P_\para^i(B)$ and~$P_\para^j(B)$ are disjoint.

The sets~$B(0)$ and~$B(1)$, defined by
\begin{equation}
  \label{eq:6}
  B(0)
  \=
  \{ z \in K \: |z| < 1 \}
  \text{ and }
  B(1)
  \=
  \{ z \in K \: |z - 1| < 1 \},
\end{equation}
are disjoint and their union contains~$\cK(P_\para)$ (\cref{l:preperiodic-or-wandering}).
The \emph{itinerary for~$P_\para$} of a point~$x$ of~$\cK(P_\para)$ is the sequence~$\theta_0 \theta_1 \ldots$ in~$\{ 0, 1 \}$ such that for every~$i$ in~$\Z_{\ge 0}$, the point~$P_\para^i(x)$ belongs to~$B(\theta_i)$.
A component of~$\cK(P_\para)$ is wandering if and only if the common itinerary of its points is not preperiodic, and in that case it contains (the trace in~$K$ of) a nontrivial wandering \textsc{Fatou} component of the same diameter (\cref{l:preperiodic-or-wandering}).

\begin{custtheo}{Theorem}
  \label{t:main}
  Let~$(\ell_s)_{s = 0}^{+\infty}$ be a strictly increasing sequence of integers satisfying ${\ell_0 = 0}$ and put
  \begin{equation}
    \label{eq:7}
    t
    \=
    -\frac{p}{(p - 1)^2} + \frac{p^{2p - 1} - p + 1}{(p - 1)^2 p^{2p + 1}(p^{2p - 2} - 1)} \sum_{s = 0}^{+\infty} \frac{1}{p^{q\ell_s + s}} \left( 1 - \frac{\kappa}{p^{2(p - 1)^3(\ell_{s + 1} - \ell_s)}} \right).
  \end{equation}
  Given~$k$ in~$\Z_{\ge 0}$, let~$s$ in~$\Z_{\ge 0}$ be uniquely determined by
  \begin{equation}
    \label{eq:8}
    \ell_s(p - 1)^2 - 1
    \le
    k
    <
    \ell_{s + 1}(p - 1)^2 - 1
  \end{equation}
  and put
  \begin{equation}
    \label{eq:9}
    M_k
    \=
    2k + 2 + (2p - 3) \ell_s.
  \end{equation}
  Furthermore, put~$m_0 \= 2p + 1$ and, for each~$k$ in~$\Z_{> 0}$,
  \begin{equation}
    \label{eq:10}
    m_k
    \=
    (p - 1)M_{k - 1} + 2p + 1.
  \end{equation}
  Then, for every~$\para_0$ in~$\pspace$ there is a parameter~$\para$ in~$\pspace$ satisfying ${|\para - \para_0| = |\para_0|^{-(M_0 - 1)}}$ and a point~$x_0$ in~$\cK(P_\para)$ whose itinerary for~$P_\para$ is
  \begin{equation}
    \label{eq:11}
    \underbrace{0 \ldots 0}_{m_0}
    \underbrace{1 \ldots 1}_{M_0}
    \underbrace{0 \ldots 0}_{m_1}
    \underbrace{1 \ldots 1}_{M_1}
    \ldots
  \end{equation}
  and such that the ball~$\{ z \in K \: |z - x_0| \le |\para|^t \}$ is a nontrivial wandering component of~$\cK(P_\para)$.
\end{custtheo}

Since the number~$t$ in the \cref{t:main} with ${(\ell_s)_{s = 0}^{+\infty} = (s)_{s = 0}^{+\infty}}$ is rational and ${M_k \to +\infty}$ as ${k \to +\infty}$, \cref{t:wandering-rationally} is a direct consequence of the \cref{t:main}, see \cref{l:preperiodic-or-wandering}.

The proof of \cref{t:wandering-irrationally} relies on the following lemma.

\begin{lemm}
  \label{l:wandering-diadically}
  Denote by~$\fT_0$ the set of real numbers of the form
  \begin{equation}
    \label{eq:12}
    \sum_{s = 0}^{+\infty} \frac{1}{p^{q\ell_s + s}} \left( 1 - \frac{\kappa}{p^{2(p - 1)^3(\ell_{s + 1} - \ell_s)}} \right),
  \end{equation}
  where~$(\ell_s)_{s = 0}^{+\infty}$ is a strictly increasing sequence in~$\Z_{\ge 0}$ satisfying ${\ell_0 = 0}$.
  Then, there are rational numbers~$\alpha$ and~$\alpha'$ satisfying ${\alpha' \neq 0}$ and, for every sequence~$(\beta(m))_{m = 0}^{+\infty}$ in~$\{0, 1\}$,
  \begin{equation}
    \label{eq:13}
    \alpha + \alpha' \sum_{m = 0}^{+\infty} \frac{\beta(m)}{p^{2q(q + 1)m}} \in \fT_0.
  \end{equation}
\end{lemm}

\begin{proof}
  Put
  \begin{equation}
    \label{eq:14}
    P
    \=
    p^{q(q + 1)},
    Q
    \=
    p^{2(p - 1)^3(q + 1)},
    \text{ and }
    E
    \=
    \left( 1 - \frac{\kappa}{p^{2(p - 1)^3}} \right) \sum_{r = 0}^{q - 2} \frac{1}{p^{(q + 1)r}} + \frac{1}{p^{q^2 - 1}}.
  \end{equation}
  The number~$E$ is rational and satisfies ${E > 0}$ because ${\kappa \le 1}$.
  Moreover, put
  \begin{equation}
    \label{eq:15}
    F
    \=
    \frac{1}{E} \cdot \frac{\kappa}{p^{q^2 - 1} p^{2(p - 1)^3q}},
  \end{equation}
  \begin{equation}
    \label{eq:16}
    \alpha
    \=
    E \left( 1 - \frac{F}{Q^2} \right) \frac{P^{2m}}{P^{2m} - 1},
    \text{ and }
    \alpha'
    \=
    E \left( \left( 1 - \frac{F}{Q} \right) \left( 1 + \frac{1}{P} \right) - \left( 1 - \frac{F}{Q^2} \right) \right).
  \end{equation}
  The numbers~$F$, ${\alpha}$, and~$\alpha'$ are rational.
  From
  \begin{equation}
    \label{eq:17}
    E
    >
    0,
    \kappa
    \le
    1,
    p^{2(p - 1)^3}E
    \ge
    p^{2(p - 1)^3} - \kappa
    \ge
    1,
    FP
    =
    \frac{\kappa p^{q + 1}}{p^{2(p - 1)^3q}E}
    \le
    \frac{p^{q + 1}}{p^{2(p - 1)^3(q - 1)}}
    \le
    \frac{1}{p^{q - 3}}
    \le
    1,
  \end{equation}
  and
  \begin{equation}
    \label{eq:18}
    \left( 1 - \frac{F}{Q} \right) \left( 1 + \frac{1}{P} \right)
    \ge
    \left( 1 - \frac{1}{QP} \right) \left( 1 + \frac{1}{P} \right)
    >
    1
    >
    1 - \frac{F}{Q^2},
  \end{equation}
  it follows that ${\alpha' > 0}$.

  Let~$(\beta(m))_{m = 0}^{+\infty}$ be a sequence in~$\{0, 1\}$ and define the sequence~$(u_v)_{v = 0}^{+\infty}$ in~$\Z_{\ge 0}$ recursively by ${u_0 \= 0}$ and, for~$v$ in~$\Z_{> 0}$, by
  \begin{equation}
    \label{eq:19}
    u_{v + 1}
    \=
    \begin{cases}
      u_v + 2
      & \text{if~$u_v$ is even and ${\beta(u_v/2) = 0}$};
      \\
      u_v + 1
      & \text{otherwise}.
    \end{cases}
  \end{equation}
  Note that~$(u_v)_{v = 0}^{+\infty}$ is strictly increasing and
  \begin{equation}
    \label{eq:20}
    \{ u_v \: v \in \Z_{\ge 0}\}
    =
    \{ k \in \Z_{\ge 0} \: \text{$k$ is even, or~$k$ is odd and satisfies ${\beta((k - 1)/2) = 1}$} \}.
  \end{equation}
  Let~$(\ell_s)_{s = 0}^{+\infty}$ be the sequence of integers defined, for all~$v$ in~$\Z_{\ge 0}$ and~$r$ in ${\{0, \ldots, q - 1\}}$, by
  \begin{equation}
    \label{eq:21}
    \ell_{vq + r}
    \=
    (q + 1)u_v - v + r.
  \end{equation}
  Note that ${\ell_0 = 0}$ and
  \begin{equation}
    \label{eq:22}
    \ell_{vq + r + 1}
    =
    \ell_{vq + r} +
    \begin{cases}
      1
      &
        \text{if ${r \neq q - 1}$};
      \\
      (q + 1)(u_{v + 1} - u_v) + q
      & \text{if ${r = q - 1}$}.
    \end{cases}
  \end{equation}
  In particular, $(\ell_s)_{s = 0}^{+\infty}$ is strictly increasing.
  A direct computation using the definition of~$\alpha$ and~$\alpha'$, \eqref{eq:20}, and~\eqref{eq:22}, yields
  \begin{multline}
    \label{eq:23}
    \sum_{s = 0}^{+\infty} \frac{1}{p^{q\ell_s + s}} \left( 1 - \frac{\kappa}{p^{2(p - 1)^3(\ell_{s + 1} - \ell_s)}} \right)
    \\
    \begin{aligned}
      & =
        \left( 1 - \frac{\kappa}{p^{2(p - 1)^3}} \right) \sum_{v = 0}^{+\infty} \sum_{r = 0}^{q - 2} \frac{1}{P^{u_v} p^{(q + 1)r}}
        + \sum_{v = 0}^{+\infty} \frac{1}{P^{u_v} p^{q^2 - 1}} \left( 1 - \frac{\kappa}{Q^{u_{v + 1} - u_v} p^{2(p - 1)^3 q}} \right)
      \\ & =
           E \sum_{v = 0}^{+\infty} \frac{1}{P^{u_v}} \left( 1 - \frac{F}{Q^{u_{v + 1} - u_v}} \right)
      \\ & =
           E \left( 1 - \frac{F}{Q} \right) \sum_{\substack{k \in \Z_{\ge 0} \\ k \text{ even and $\beta(k/2) = 1$, or} \\ k \text{ odd and $\beta((k - 1)/2) = 1$}}} \frac{1}{P^k}
      + E \left( 1 - \frac{F}{Q^2} \right) \sum_{\substack{k \in \Z_{\ge 0} \\ k \text{ even and $\beta(k/2) = 0$}}} \frac{1}{P^k}
      \\ & =
           E \left( 1 - \frac{F}{Q} \right) \sum_{m = 0}^{+\infty} \left( \frac{\beta(m)}{P^{2m}} + \frac{\beta(m)}{P^{2m + 1}} \right) + E \left( 1 - \frac{F}{Q^2} \right) \sum_{m = 0}^{+\infty} \frac{1 - \beta(m)}{P^{2m}}
      \\ & =
           \alpha + \alpha' \sum_{m = 0}^{+\infty} \frac{\beta(m)}{P^{2m}}.
           \qedhere
    \end{aligned}
  \end{multline}
\end{proof}

\begin{proof}[Proof of \cref{t:wandering-irrationally} supposing the \cref{t:main}]
  Fix~$r_0$ in~$\sD$ distinct from~$1$ and put
  \begin{equation}
    \label{eq:24}
    \fD
    \=
    \{ \log_{r_0} r \: r \in \sD \}.
  \end{equation}
  The set~$\fD$ is a $\Q$\nobreakdash-vector space contained in~$\R$.

  Suppose that for every~$\para$ in~$\pspace$, the diameter of every nontrivial wandering \textsc{Fatou} component of~$P_\para$ were in~$\fD$.
  Combined with the \cref{t:main} and \cref{l:wandering-diadically}, this would yield ${\fT_0 \subseteq \fD}$, see \cref{l:preperiodic-or-wandering}.
  To arrive to a contradiction, it is thus sufficient to prove that the $\Q$\nobreakdash-vector space~$\fT$ generated by~$\fT_0$ is~$\R$.
  The rest of the proof achieves this by proving ${[0, 1) \subseteq \fT}$.

  Let~$\tau$ in~$[0, 1)$ be given, put ${B \= p^{2q(q + 1)}}$, and let~$(d(m))_{m = 1}^{+\infty}$ be a sequence of digits of the expansion of~$\tau$ in base~$B$.
  So, for every~$m$ in~$\Z_{> 0}$, the number~$d(m)$ belongs to~$\{0, \ldots, B - 1 \}$ and
  \begin{equation}
    \label{eq:25}
    \tau
    =
    \sum_{m = 0}^{+\infty} \frac{d(m)}{B^m}.
  \end{equation}
  Moreover, for all~$j$ in~$\{ 1, \ldots, B - 1 \}$ and~$m$ in~$\Z_{\ge 0}$, put ${\beta^{(j)}(m) \= 1}$ if ${j \le d(m)}$ and ${\beta^{(j)}(0) \= 0}$ otherwise, and define
  \begin{equation}
    \label{eq:26}
    \tau^{(j)}
    \=
    \sum_{m = 0}^{+\infty} \frac{\beta^{(j)}(m)}{B^m}.
  \end{equation}
  Then, for every~$m$ in~$\Z_{\ge 0}$,
  \begin{equation}
    \label{eq:27}
    \beta^{(1)}(m) + \cdots + \beta^{(B - 1)}(m)
    =
    d(m),
  \end{equation}
  and thus
  \begin{equation}
    \label{eq:28}
    \tau^{(1)} + \cdots + \tau^{(B - 1)}
    =
    \tau.
  \end{equation}
  On the other hand, for each~$j$ in~$\{ 1, \ldots, B - 1 \}$ the number~$\tau^{(j)}$ belongs to~$\fT_0$ by \cref{l:wandering-diadically}.
  Together with~\eqref{eq:28}, this implies that~$\tau$ belongs to~$\fT$.
  This proves ${[0, 1) \subseteq \fT}$ and, hence ${\fT = \fD = \R}$.
  This contradicts the hypothesis that~$\sD$ is a proper subgroup of~$\R_{> 0}$.
\end{proof}

\section{Preliminaries}
\label{s:preliminaries}
For~$z_0$ in~$K$ and~$r$ in~$\R_{> 0}$, the set ${\{ z \in K \: |z - z_0| < r \}}$ is a \emph{disk of~$K$} and ${\{ z \in K \: |z - z_0| \le r \}}$ a \emph{ball of~$K$}.
These sets coincide when~$r$ is outside~$|K^{\times}|$.
Two disks or balls of~$K$ are disjoint or nested.
The image of a disk (resp. ball) by a polynomial is a disk (resp. ball), see for example~\cite[\S1.3.2]{Riv03c}.

For~$\para$ in~$\pspace$, the following lemma gathers basic properties of~$\cK(P_\para)$.

\begin{lemm}
  \label{l:preperiodic-or-wandering}
  For every~$\para$ in~$\pspace$, the inclusion ${\cK(P_\para) \subseteq B(0) \cup B(1)}$ holds.
  Furthermore, the following properties hold for all~$z_0$ in~$\cK(P_\para)$ and~$r$ in~$\R_{> 0}$ such that the ball~$B$, defined by
  \begin{equation}
    \label{eq:29}
    B
    \=
    \{ z \in K \: |z - z_0| \le r \},
  \end{equation}
  is a component of~$\cK(P_\para)$.
  \begin{enumerate}
  \item
    All points in~$B$ have the same itinerary for~$P_\para$.
  \item
    If the common itinerary of points in~$B$ is not preperiodic, then~$B$ is wandering.
    Furthermore, $\{ z \in K \: |z - z_0| < r \}$ is a nontrivial wandering component of the \textsc{Fatou} set of~$P$ in~$\PK$ and the intersection of~$\PK$ and a nontrivial wandering component of the \textsc{Fatou} set of~$P$ in~$\PKber$.
  \end{enumerate}
\end{lemm}

\begin{proof}
  Let~$\aB(0)$ and~$\aB(1)$ be the balls of~$\PKber$ whose intersection with~$\PK$ are~$B(0)$ and~$B(1)$, respectively, and put
  \begin{equation}
    \label{eq:30}
    \aK(P_\para)
    \=
    \{ x \in \PKber \: (|P_\para^n(x)|)_{n = 1}^{+\infty} \text{ is bounded} \}.
  \end{equation}
  The equality ${\aK(P_\para) \cap K = \cK(P_\para)}$ holds.

  To prove the first assertion, note that, for every~$x$ in~$\PKber$ satisfying ${|x| > 1}$,
  \begin{equation}
    \label{eq:31}
    |P_\para(x)|
    =
    |\para| \cdot |x|^{p + 1}
    >
    |\para| \cdot |x|.
  \end{equation}
  This implies $|P_\para^n(x)| \to +\infty$ as $n \to +\infty$ and ${\aK(P_\para) \subseteq \{ x \in \PKber \: |x| \le 1 \}}$.
  On the other hand, every~$x$ in ${\PKber \setminus \aB(1)}$, the equality ${|x| = 1}$ implies ${|P_\para(x)| = |\para| > 1}$ and hence that~$x$ is outside~$\aK(P_\para)$.
  This proves ${\aK(P_\para) \subseteq \aB(0) \cup \aB(1)}$ and thus ${\cK(P_\para) \subseteq B(0) \cup B(1)}$.

  To prove item~$1$, suppose that~$B$ contained points with distinct itineraries~$\theta_0 \theta_1 \ldots$ and~$\theta_0' \theta_1' \ldots$ and let~$n$ be the smallest index satisfying ${\theta_n \neq \theta_n'}$.
  Then~$P_\para^n(B)$ would intersect~$B(0)$ and~$B(1)$ and hence
  \begin{equation}
    \label{eq:32}
    \{ z \in K \: |z| \le 1 \}
    \subseteq
    P_\para^n(B)
    \subseteq
    \cK(P_\para)
    \subseteq
    B(0) \cup B(1),
  \end{equation}
  which is absurd.

  To prove item~$2$, suppose~$B$ were not wandering and let~$i$ and~$j$ be integers satisfying
  \begin{equation}
    \label{eq:33}
    i
    >
    j
    \ge
    0
    \text{ and }
    P_\para^i(B) \cap P_\para^j(B)
    \neq
    \emptyset.
  \end{equation}
  Then, there would be~$x$ in~$B$ such that~$P_\para^{j - i}(x)$ belongs to~$P_\para^i(B)$.
  The itinerary of~$x$ would be preperiodic by item~$1$, which is absurd.
  To prove the second assertion, put ${D \= \{ z \in K \: |z - z_0| < r \}}$.
  All the iterates of~$D$ are contained in the bounded set~$\cK(P_\para)$, so~$D$ is contained in a component of the \textsc{Fatou} set of~$P_\para$ in~$\PK$.
  Let~$\aD$ be the ball of~$\PKber$ whose intersection with~$\PK$ is~$D$.
  If~$\aD$ were not a \textsc{Fatou} component of~$P_\para$ in~$\PKber$, then there would be a real number~$r'$ satisfying ${r' > r}$ such that~$\aD$ contained the disk~$D'$, defined by~$D' \= \{ z \in K \: |z - z_0| < r' \}$.
  Since~$B$ is a component of~$\cK(P_\para)$, there is a point~$w$ of~$D'$ outside~$\cK(P_\para)$, so ${|P_\para^n(w)| \to +\infty}$ as ${n \to +\infty}$.
  Since for each~$n$ in~$\Z_{> 0}$ the set~$P_\para^n(D')$ is a disk containing~$P_\para^n(w)$ and intersecting the bounded set~$\cK(P_\para)$, it follows that ${\bigcup_{n = 1}^{+\infty} P_\para^n(D') = K}$.
  This contradicts the hypothesis that~$D'$ is contained in a \textsc{Fatou} component and shows that~$\aD$ is a component of the \textsc{Fatou} set of~$P_\para$ in~$\PKber$.
  It follows that~$D$ is a component of the \textsc{Fatou} set of~$P_\para$ in~$\PK$ \cite{0Riv0412}.
  It remains to prove that the wandering \textsc{Fatou} component~$\aD$ is nontrivial.
  Suppose this were not the case, so there is~$x$ in~$\aD$ whose orbit is asymptotic to a periodic orbit in~$\PKber$.
  So, there is~$n_0$ in~$\Z_{> 0}$ such that the orbit of~$x$ under~$P_\para^{n_0}$ converges to a fixed point~$x_0$ of~$P_\para^{n_0}$.
  Since~${x_0}$ belongs to~$\aK(P_\para)$, ${\aK(P_\para) \subseteq \aB(0) \cup \aB(1)}$, and the sets~$\aB(0)$ and~$\aB(1)$ are open and disjoint, it follows that there is an iterate of~$x$ whose itinerary for~$P_\para$ for is the same as that of~$x_0$.
  Since~$x_0$ is periodic, this would imply that the itinerary of~$x$ is preperiodic, which is absurd.
\end{proof}

For each~$\para$ in~$\pspace$, put ${\rho(\para) \= |\para|^{-\frac{1}{p - 1}}}$.
Then,
\begin{equation}
  \label{eq:34}
  |p|
  \le
  \rho(\para)
  <
  1
  \text{ and }
  |\para| \rho(\para)^p
  =
  \rho(\para).
\end{equation}
Moreover, for each~$m$ in~$\Z_{\ge 0}$, put
\begin{equation}
  \label{eq:35}
  \rho_m(\para)
  \=
  |\para|^{-\frac{1}{p - 1} \left( 1 - \frac{1}{p^m} \right)}.
\end{equation}
The sequence~$(\rho_m(\para))_{m = 0}^{+\infty}$ is strictly decreasing and converges to~$\rho(\para)$, ${\rho_0(\para) = 1}$, and, for each~$m$ in~$\Z_{> 0}$, the equality ${|\para| \rho_m(\para)^p = \rho_{m - 1}(\para)}$ holds.

Note that ${P_\para(0) = 0}$ and that, for every~$x$ in~$B(0)$, the equality ${|P_\para(x)| = |\para| \cdot |x|^p}$ holds.
Hence, for all~$m$ in~$\Z_{> 0}$, ${i}$ in~$\{0, \ldots, m\}$, and~$x$ in~$B(0)$ satisfying ${|x| = \rho_m(\para)}$,
\begin{equation}
  \label{eq:36}
  |P_\para^i(x)|
  =
  \rho_{m - i}(\para).
\end{equation}

\begin{lemm}
  \label{l:contraction-expansion}
  For every parameter~$\para$ in~$\pspace$, the following properties hold.
  \begin{enumerate}
  \item
    For all~$m$ in~$\Z_{> 0}$ and~$x$ and~$x'$ in~$K$ satisfying
    \begin{equation}
      \label{eq:37}
      |x|
      =
      \rho_m(\para)
      \text{ and }
      |x - x'|
      <
      \rho_m(\para)^{\frac{p}{p - 1}},
    \end{equation}
    \begin{equation}
      \label{eq:38}
      |P_\para(x) - P_\para(x')|
      =
      \rho_{m - 1}(\para)|x - x'|
      <
      \rho_{m - 1}(\para)^{\frac{p}{p - 1}}.
    \end{equation}
  \item
    For all~$m$ in~$\Z_{> 0}$ and~$x$ and~$x'$ in~$K$ satisfying
    \begin{equation}
      \label{eq:39}
      |x|
      =
      \rho_m(\para)
      \text{ and }
      \rho_m(\para)^{\frac{p}{p - 1}}
      <
      |x - x'|
      <
      \rho_m(\para),
    \end{equation}
    \begin{equation}
      \label{eq:40}
      |P_\para(x) - P_\para(x')|
      =
      |\para| \cdot |x - x'|^p.
    \end{equation}
  \item
    For all~$y$ and~$y'$ in~$B(1)$, the equality ${|P_\para(y) - P_\para(y')| = |\para| \cdot |y - y'|}$ holds.
  \end{enumerate}
\end{lemm}

\begin{proof}
  \hfill

  \partn{1}
  Setting $\varepsilon \= x' - x$,
  \begin{equation}
    \label{eq:n1}
    P_\para(x') - P_\para(x)
    =
    \para(p \varepsilon x^{p - 1} + \cdots + p \varepsilon^{p - 1}x + \varepsilon^p)
    +
    \left( 1 - \para \right) ((p + 1) \varepsilon x^p + \cdots + \varepsilon^{p + 1}).
  \end{equation}
  Since, by hypothesis,
  \begin{equation}
    \label{eq:41}
    |\varepsilon|
    =
    |x - x'|
    <
    \rho_m(\para)^{\frac{p}{p - 1}}
    <
    \rho_m(\para)
    =
    |x|,
  \end{equation}
  ${|x'| = |x|}$ and, by the definition of~$\pspace$,
  \begin{multline}
    \label{eq:42}
    \left| P_\para(x') - P_\para(x) - \left( 1 - \para \right) (p + 1) \varepsilon x^p \right|
    \\ \le
    |\para| \cdot |\varepsilon| \max \left\{ |p| \cdot \rho_m(\para)^{p - 1}, |\varepsilon|^{p - 1}, |\varepsilon| \rho_m(\para)^{p - 1} \right\}
    \\ <
    |\para| \cdot |\varepsilon| \rho_m(\para)^p
    =
    \left| \left( 1 - \para \right) (p + 1) \varepsilon x^p \right|.
  \end{multline}
  Thus,
  \begin{equation}
    \label{eq:43}
    \left| P_\para(x') - P_\para(x) \right|
    =
    |\para| \cdot |\varepsilon| \rho_m(\para)^p
    =
    \rho_{m - 1}(\para) |x' - x|.
  \end{equation}

  \partn{2}
  As in item~$1$, put~$\varepsilon \= x' - x$, so ${|x'| = |x|}$, ${|\varepsilon| < \rho_m(\para)}$, and, by~\eqref{eq:n1},
  \begin{equation}
    \label{eq:44}
    \left| P_\para(x') - P_\para(x) - \para \varepsilon^p \right|
    \le
    |\para| \cdot |\varepsilon| \max \{ |p| \rho_m(\para)^{p - 1}, \rho_m(\para)^p \}.
  \end{equation}
  On the other hand, the hypothesis ${|\varepsilon| > \rho_m(\para)^{\frac{p}{p - 1}}}$ implies
  \begin{equation}
    \label{eq:45}
    |p| \rho_m(\para)^{p - 1}
    <
    \rho_m(\para)^p
    <
    |\varepsilon|^{p - 1}.
  \end{equation}
  Combined with~\eqref{eq:44}, this implies
  \begin{equation}
    \label{eq:46}
    \left| P_\para(x') - P_\para(x) - \para \varepsilon^p \right|
    <
    |\para| \cdot |\varepsilon|^p
    \text{ and }
    \left| P_\para(x') - P_\para(x) \right|
    =
    |\para| \cdot |\varepsilon|^p
    =
    |\para| \cdot |x - x'|^p.
  \end{equation}

  \partn{3}
  The assertion is a direct consequence of~\eqref{eq:n1} with~$y = x$, $y' = x'$, and~$\varepsilon = y - y'$.
\end{proof}

\section{Realization of itineraries}
\label{s:itinerary-realization}
This section proves the following sufficient condition for an itinerary to be realized by a point and a polynomial in~$(P_\para)_{\para \in \pspace}$.

\begin{prop}
  \label{p:itinerary-realization}
  Let~$(m_i)_{i = 0}^{+\infty}$ and~$(M_i)_{i = 0}^{+\infty}$ be sequences in~$\Z_{> 0}$ satisfying~$M_0 \ge 2$ and, for every~$i$ in~$\Z_{\ge 0}$,
  \begin{equation}
    \label{eq:47}
    M_{i + 1}
    \ge
    M_i + 2
    \text{ and }
    M_i - \frac{m_{i + 1}}{p - 1} + \frac{p}{(p - 1)^2}\left(1 - \frac{1}{p^{m_{i + 1}}}\right)
    <
    0.
  \end{equation}
  Then, for every~$\para_0$ in~$\pspace$ there is a parameter~$\para$ in~$\pspace$ satisfying ${|\para - \para_0| = |\para_0|^{-(M_0 - 1)}}$ and a point~$x_0$ in~$\cK(P_\para)$ whose itinerary for~$P_\para$ is
  \begin{equation}
    \label{eq:48}
    \underbrace{0 \ldots 0}_{m_0}
    \underbrace{1 \ldots 1}_{M_0}
    \underbrace{0 \ldots 0}_{m_1}
    \underbrace{1 \ldots 1}_{M_1}
    \ldots
  \end{equation}
\end{prop}

The proof of this proposition is a streamlined version of the arguments in~\cite{Ben02a}.
It is given in~\cref{ss:proof-itinerary-realization}, after some preliminary considerations in~\cref{ss:phase-parameter}.

\subsection{Phase parameter relation}
\label{ss:phase-parameter}
The following lemmas compare orbits for polynomials in~$(P_\para)_{\para \in \pspace}$ with different parameters.

\begin{lemm}
  \label{l:perturbation-0}
  Given~$M$ in~$\Z_{> 0}$, let~$y$ and~$y'$ in~$B(1)$ be such that ${|y - 1| \le |\para|^{-M}}$ and ${|y - y'| \le |\para|^{-M}}$.
  Then, for all parameters~$\para$ and~$\para'$ in~$\pspace$ satisfying
  \begin{equation}
    \label{eq:49}
    |\para|^{-1} \cdot |\para - \para'|
    =
    |y - y'|,
  \end{equation}
  \begin{equation}
    \label{eq:50}
    |P_\para^M(y) - P_{\para'}^M(y')|
    =
    |\para|^{M - 1} \cdot |\para - \para'|.
  \end{equation}
\end{lemm}

\begin{proof}
  The goal is to prove that, for every~$i$ in~$\{0, \ldots, M\}$,
  \begin{equation}
    \label{eq:51}
    |P_\para^i(y) - 1|
    \le
    |\para|^{-(M - i)}
    \text{ and }
    |P_\para^i(y) - P_{\para'}^i(y')|
    =
    |\para|^{i - 1} \cdot |\para - \para'|.
  \end{equation}
  The desired assertion follows from this with ${i = M}$.

  The hypotheses imply~\eqref{eq:51} with ${i = 0}$.
  Let~$i$ in~$\{0, \ldots, M - 1\}$ such that~\eqref{eq:51} holds.
  By \cref{l:contraction-expansion}(3),
  \begin{equation}
    \label{eq:52}
    |P_\para^{i + 1}(y) - 1|
    =
    |\para| \cdot |P_\para^i(y) - 1|
    \le
    |\para|^{-(M - (i + 1))}
  \end{equation}
  and
  \begin{equation}
    \label{eq:53}
    |P_{\para'}(P_\para^i(y)) - P_{\para'}^{i + 1}(y')|
    =
    |\para| \cdot |P_\para^i(y) - P_{\para'}^i(y')|
    =
    |\para|^i \cdot |\para - \para'|.
  \end{equation}
  On the other hand, putting ${\why \= P_\para^i(y)}$,
  \begin{equation}
    \label{eq:54}
    |P_\para(\why) - P_{\para'}(\why)|
    =
    |\para - \para'| \cdot |\why^p - \why^{p + 1}|
    \le
    |\para - \para'| \cdot |\para|^{-(M - i)}.
  \end{equation}
  Combined with~\eqref{eq:53} this implies the equality in~\eqref{eq:51} with~$i$ replaced by~$i + 1$.
  This completes the proof that~\eqref{eq:51} holds for every~$i$ in~$\{0, \ldots, M\}$ and of the lemma.
\end{proof}

\begin{lemm}
  \label{l:perturbation-1}
  Let~$m$ be in~$\Z_{> 0}$ and let~$x$ and~$x'$ in~$B(0)$ be such that ${|x| = \rho_m(\para)}$ and ${|x - x'| < \rho_m(\para)^{\frac{p}{p - 1}}}$.
  Then, for all parameters~$\para$ and~$\para'$ in~$\pspace$ satisfying
  \begin{equation}
    \label{eq:55}
    \rho_1(\para) \ldots \rho_{m - 1}(\para) \cdot |x - x'|
    <
    |\para|^{-1} \cdot |\para - \para'|
    <
    \rho_m(\para)^{\frac{p}{p - 1}},
  \end{equation}
  \begin{equation}
    \label{eq:56}
    |P_\para^m(x) - P_{\para'}^m(x')|
    =
    |\para|^{-1} \cdot |\para - \para'|.
  \end{equation}
\end{lemm}

\begin{proof}
  For every~$i$ in~$\{0, \ldots, m \}$, the equality ${|P_\para^i(x)| = \rho_{m - i}(\para)}$ holds by~\eqref{eq:36}.

  The first step is to prove that, for each~$i$ in~$\{1, \ldots, m \}$,
  \begin{equation}
    \label{eq:A}
    |P_\para^i(x) - P_{\para'}^i(x')|
    \le
    \max\{ \rho_{m - i}(\para) \ldots \rho_{m - 1}(\para) |x - x'|, \rho_{m - i}(\para)|\para|^{-1} \cdot |\para - \para'| \}.
  \end{equation}
  To prove this assertion for~$i = 1$, observe that, by the hypothesis ${|x - x'| < \rho_m(\para)^{\frac{p}{p - 1}}}$ and \cref{l:contraction-expansion}(1),
  \begin{equation}
    \label{eq:A1}
    |P_\para(x) - P_\para(x')| = \rho_{m - 1}(\para) |x - x'|.
  \end{equation}
  On the other hand, $P_\para(x') - P_{\para'}(x') = (\para - \para')((x')^p - (x')^{p + 1})$ yields
  \begin{equation}
    \label{eq:A2}
    |P_\para(x') - P_{\para'}(x')|
    =
    |\para - \para'| \cdot |x'|^p
    =
    \rho_{m - 1}(\para) |\para|^{-1} \cdot |\para - \para'|.
  \end{equation}
  Together with~\eqref{eq:A1} and the ultrametric inequality, this implies~\eqref{eq:A} with~$i = 1$.
  Let~$i$ in ${\{1, \ldots, m - 1 \}}$ be such that~\eqref{eq:A} holds.
  The hypotheses imply
  \begin{equation}
    \label{eq:57}
    |x - x'|
    <
    \rho_m(\para)^{\frac{p}{p - 1}}
    <
    \rho_{m - i}(\para)^{\frac{p}{p - 1}}
    \text{ and }
    |\para|^{-1} \cdot |\para - \para'|
    <
    \rho_m(\para)^{\frac{p}{p - 1}}
    <
    \rho_{m - i}(\para)^{\frac{p}{p - 1}},
  \end{equation}
  so by~\eqref{eq:A},
  \begin{equation}
    \label{eq:58}
    |P_\para^i(x) - P_{\para'}^i(x')|
    <
    \rho_{m - i}(\para)^{\frac{p}{p - 1}}.
  \end{equation}
  Using~\eqref{eq:A} again and \cref{l:contraction-expansion}(1) with~$m$, $x$ and~$x'$ replaced by~$m - i$, $P_\para^i(x)$ and~$P_{\para'}^i(x')$, respectively, yields
  \begin{multline}
    \label{eq:n2}
    |P_\para^{i + 1}(x) - P_\para(P_{\para'}^i(x'))|
    =
    \rho_{m - i - 1}(\para) |P_\para^i(x) - P_{\para'}^i(x')|
    \\ \le
    \max\{ \rho_{m - i - 1}(\para) \rho_{m - i}(\para) \ldots \rho_{m - 1}(\para) |x - x'|, \rho_{m - i - 1}(\para) \rho_{m - i}(\para) |\para|^{-1} \cdot |\para - \para'| \}.
  \end{multline}
  On the other hand, $P_\para(y) - P_{\para'}(y) = (\para - \para') (y^p - y^{p + 1})$ with~$y = P_{\para'}^i(x')$ yields
  \begin{equation}
    \label{eq:n3}
    |P_\para(P_{\para'}^i(x')) - P_{\para'}^{i + 1}(x')|
    =
    |\para - \para'| \cdot |P_{\para'}^i(x')|^p
    =
    \rho_{m - i - 1}(\para) |\para|^{-1} \cdot |\para - \para'|.
  \end{equation}
  Together with~\eqref{eq:n2} and the ultrametric inequality, this implies~\eqref{eq:A} with~$i$ replaced by ${i + 1}$.
  By induction, \eqref{eq:A} holds for every~$i$ in~$\{1, \ldots, m \}$.

  By the hypothesis~\eqref{eq:55} and by~\eqref{eq:A} with $i = m - 1$,
  \begin{multline}
    \label{eq:59}
    |P_\para^{m - 1}(x) - P_{\para'}^{m - 1}(x')|
    \le
    \max\{ \rho_1(\para) \ldots \rho_{m - 1}(\para) |x - x'|, \rho_1(\para) |\para|^{-1} \cdot |\para - \para'| \}
    \\ <
    |\para|^{-1} \cdot |\para - \para'|.
  \end{multline}
  Using hypothesis~\eqref{eq:55}, again, and \cref{l:contraction-expansion}(1) with ${m = 1}$ and~$x$ and~$x'$ replaced by~$P_\para^{m - 1}(x)$ and~$P_{\para'}^{m - 1}(x')$, respectively,
  \begin{equation}
    \label{eq:60}
    |P_\para^m(x) - P_\para(P_{\para'}^{m - 1}(x'))|
    =
    |P_\para^{m - 1}(x) - P_{\para'}^{m - 1}(x')|
    <
    |\para|^{-1} \cdot |\para - \para'|.
  \end{equation}
  Together with~\eqref{eq:n3} with $i = m - 1$ and with the ultrametric inequality, this implies the desired equality.
\end{proof}

\subsection{Realization of itineraries}
\label{ss:proof-itinerary-realization}
The proof of \cref{p:itinerary-realization} recursively defines a sequence of parameters for which a chosen point follows the prescribed itinerary up to the moment it maps to the fixed point~$1$.
This sequence of parameters converges to a limit for which the point follows the prescribed itinerary indefinitely.
The next lemma provides the inductive step.

\begin{lemm}
  \label{l:inductive-realization}
  Let~$\para_*$ in~$\pspace$, $n$ in~$\Z_{> 0}$, $x_0$ in~$P_{\para_*}^{-n}(1)$, and~$\varepsilon$ in~$(0, 1)$ be such that, for all parameters~$\para$ and~$\para'$ in~$\pspace$ satisfying~$|\para - \para_*| < \varepsilon$ and~$|\para' - \para_*| < \varepsilon$,
  \begin{equation}
    \label{eq:61}
    |P_\para^n(x_0) - P_{\para'}^n(x_0)|
    =
    |\para_*|^{-1} \cdot |\para - \para'|.
  \end{equation}
  Moreover, let~$\theta_0 \theta_1 \ldots \theta_{n - 1}$ be the sequence in~$\{0, 1 \}$ such that the itinerary of~$x_0$ for~$P_{\para_*}$ is
  \begin{equation}
    \label{eq:62}
    \theta_0 \theta_1 \ldots \theta_{n - 1} 11 \ldots
  \end{equation}
  Then, for all~$M$ and~$m$ in~$\Z_{> 0}$ satisfying
  \begin{equation}
    \label{eq:63}
    |\para_*|^{-(M - 1)} < \varepsilon
    \text{ and }
    |\para_*|^M \cdot \rho_1(\para) \ldots \rho_{m - 1}(\para)
    <
    1,
  \end{equation}
  there is~$\para_*'$ in~$\pspace$ satisfying ${|\para_*' - \para_*| = |\para_*|^{-(M - 1)}}$, such that the itinerary of~$x_0$ for~$P_{\para_*'}$ is
  \begin{equation}
    \label{eq:64}
    \theta_0 \theta_1 \ldots \theta_{n - 1}
    \underbrace{1 \ldots 1}_{M}
    \underbrace{0 \ldots 0}_{m}
    11 \ldots
  \end{equation}
  and such that, putting~$n' \= n + M + m$, for all~$\para$ and~$\para'$ in~$\pspace$ satisfying
  \begin{equation}
    \label{eq:65}
    \max \{ |\para - \para_*'|, |\para' - \para_*'| \}
    <
    \rho_m(\para)^{\frac{p}{p - 1}} |\para_*'|^{-(M - 1)},
  \end{equation}
  \begin{equation}
    \label{eq:66}
    |P_\para^{n'}(x_0) - P_{\para'}^{n'}(x_0)|
    =
    |\para_*'|^{-1} \cdot |\para - \para'|.
  \end{equation}
\end{lemm}

\begin{proof}
  Let~$\phi_0(\para)$ be the polynomial with coefficients in~$K$ defined by ${\phi_0(\para) \= P_\para^{n + M}(x_0)}$.
  Then, ${\phi_0(\para_*) = 1}$.
  The hypotheses imply that, for all~$\para$ and~$\para'$ in~$\pspace$ satisfying
  \begin{equation}
    \label{eq:67}
    \max\{ |\para - \para_*|, |\para' - \para_*| \}
    \le
    |\para_*|^{-(M - 1)},
  \end{equation}
  \begin{equation}
    \label{eq:68}
    |P_\para^n(x_0) - P_{\para'}^n(x_0)|
    =
    |\para_*|^{-1} \cdot |\para - \para'|
    \le
    |\para_*|^{-M}.
  \end{equation}
  Taking ${\para' = \para_*}$, this yields ${|P_\para^n(x_0) - 1| \le |\para_*|^{-M}}$.
  Hence, \cref{l:perturbation-0} with $y = P_\para^n(x_0)$ and $y' = P_{\para'}^n(x_0)$ implies
  \begin{equation}
    \label{eq:AA}
    |\phi_0(\para) - \phi_0(\para')|
    =
    |\para_*|^{M - 1} \cdot |\para - \para'|.
  \end{equation}
  It follows that~$\phi_0$ maps~$\{ \para \in \pspace \: |\para - \para_*| \le |\para_*|^{-(M - 1)}\}$ univalently onto~$\{ z \in K \: |z| \le 1 \}$, see \textsc{Schwarz}' Lemma in~\cite[\S1.3.1]{Riv03c}.
  In particular, there is~$\hpara_*$ in~$\pspace$ satisfying
  \begin{equation}
    \label{eq:69}
    |\hpara_* - \para_*|
    =
    |\para_*|^{-(M - 1)}
    \text{ and }
    P_{\hpara_*}^{n + M}(x_0)
    =
    \phi_0(\hpara_*) = 0.
  \end{equation}
  Note that ${|\hpara_*| = |\para_*|}$ and that for every~$\para$ in~$\pspace$ satisfying ${|\para - \hpara_*| < |\hpara_*|^{-(M - 1)}}$ the inequality ${|\phi_0(\para)| < 1}$ holds, and the itinerary up to time~$n + M + 1$ of~$x_0$ for~$P_\para$ is
  \begin{equation}
    \label{eq:70}
    \theta_0 \theta_1 \ldots \theta_{n - 1} \underbrace{1 \ldots 1}_{M} 0.
  \end{equation}

  Recall that ${n' = n + M + m}$, let~$\phi_1(\para)$ be the polynomial with coefficients in~$K$ defined by ${\phi_1(\para) \= P_\para^{n'}(x_0)}$.
  Then, ${\phi_1(\hpara_*) = 0}$.
  Since~$\phi_0$ maps ${\{ \para \in \pspace \: |\para - \para_*| \le |\para_*|^{-(M - 1)}\}}$ univalently onto ${\{ z \in K \: |z| \le 1 \}}$, for every~$\para$ in~$\pspace$ satisfying $|\para - \hpara_*| = \rho_m(\para) |\para|^{-(M - 1)}$, the equality ${|\phi_0(\para)| = \rho_m(\para)}$ holds.
  Hence, ${|\phi_1(\para)| = |P_\para^m(\phi_0(\para))| = \rho_0(\para) = 1}$ by~\eqref{eq:36} with ${i = 0}$ and ${x = \phi_0(\para)}$.
  Since~$\phi_1$ is a polynomial, this implies
  \begin{equation}
    \label{eq:71}
    \phi_1(\{ \para \in K \: |\para - \hpara_*| \le \rho_m(\para) |\para|^{-(M - 1)} \})
    =
    \{ z \in K \: |z| \le 1 \}.
  \end{equation}
  In particular, there is~$\para_*'$ in~$\pspace$ satisfying ${|\para_*' - \hpara_*| \le \rho_m(\para) |\para|^{-(M - 1)}}$ and
  \begin{equation}
    \label{eq:72}
    P_{\para_*'}^{n + M + m}(x_0)
    =
    \phi_1(\para_*') = 1.
  \end{equation}
  In fact, ${|\para_*' - \hpara_*| = \rho_m(\para) |\para|^{-(M - 1)}}$, otherwise ${|\phi_0(\para_*')| < \rho_m(\para)}$ by~\eqref{eq:69} and thus ${|\phi_1(\para_*')| < 1}$.
  It follows that ${|\para_*'| = |\para_*|}$ and that the itinerary of~$x_0$ for~$P_{\para_*'}$ is
  \begin{equation}
    \label{eq:73}
    \theta_0 \theta_1 \ldots \theta_{n - 1}
    \underbrace{1 \ldots 1}_{M}
    \underbrace{0 \ldots 0}_{m}
    11 \ldots
  \end{equation}

  Let~$\para$ and~$\para'$ in~$\pspace$ be satisfying~\eqref{eq:65} and put ${x \= \phi_0(\para)}$ and ${x' \= \phi_0(\para')}$.
  Since
  \begin{equation}
    \label{eq:74}
    |\phi_0(\para_*')|
    =
    |P_{\para_*'}^{n + M}(x_0)|
    =
    \rho_m(\para),
  \end{equation}
  equation~\eqref{eq:AA} implies
  \begin{equation}
    \label{eq:75}
    |x - x'|
    =
    |\para_*|^{M - 1} \cdot |\para - \para'|
    <
    \rho_m(\para)^{\frac{p}{p - 1}}
    \text{ and }
    |x|
    =
    |x'|
    =
    \rho_m(\para).
  \end{equation}
  Combined with the second inequality in~\eqref{eq:63}, this implies
  \begin{equation}
    \label{eq:76}
    \rho_1(\para) \ldots \rho_{m - 1}(\para) |x - x'|
    =
    \rho_1(\para) \ldots \rho_{m - 1}(\para) \cdot |\para_*|^{M - 1} \cdot |\para - \para'|
    <
    |\para|^{-1} \cdot |\para - \para'|,
  \end{equation}
  and \cref{l:perturbation-1} yields
  \begin{equation}
    \label{eq:77}
    |P_\para^{n'}(x_0) - P_{\para'}^{n'}(x_0)|
    =
    | P_\para^m(x) - P_{\para'}^m(x')|
    =
    |\para_*'|^{-1} \cdot |\para - \para'|.
    \qedhere
  \end{equation}
\end{proof}

\begin{proof}[Proof of \cref{p:itinerary-realization}]
  Define the sequence~$(n_i)_{i = 0}^{+\infty}$ in~$\Z_{> 0}$ recursively by~$n_0 \= m_0$ and, for~$i$ in~$\Z_{> 0}$, by ${n_i \= n_{i - 1} + M_{i - 1} + m_i}$.
  Fix~$\para_0$ in~$\pspace$, put~$\varepsilon_0 \= |\para_0| \rho_{m_0}(\para_0)^{\frac{p}{p - 1}}$, and for each~$i$ in~$\Z_{> 0}$ put ${\varepsilon_i \= |\para_0|^{-(M_{i - 1} - 1)} \rho_{m_i}(\para_0)^{\frac{p}{p - 1}}}$.
  Since for each~$m$ in~$\Z_{> 0}$,
  \begin{equation}
    \label{eq:78}
    \log_{|\para|} \rho_m(\para_0)^{\frac{p}{p - 1}}
    >
    \log_{|\para|} \rho(\para_0)^{\frac{p}{p - 1}}
    =
    -\frac{p}{(p - 1)^2}
    \ge
    -2,
  \end{equation}
  for every~$i$ in~$\Z_{\ge 0}$ the inequality ${|\para_0|^{-(M_i - 1)} < \varepsilon_i}$ holds.
  On the other hand, the hypotheses imply that, for every~$i$ in~$\Z_{\ge 0}$,
  \begin{equation}
    \label{eq:79}
    |a_0|^{M_i} \rho_1(\para_0) \ldots \rho_{m_{i + 1} - 1}(\para_0)
    =
    |\para_0|^{M_i - \frac{m_{i + 1}}{p - 1} + \frac{p}{(p - 1)^2}\left(1 - \frac{1}{p^{m_{i + 1}}}\right)}
    <
    1.
  \end{equation}

  Let~$x_0$ be a point in~$K$ of norm~$\rho_{m_0}(\para_0)$ such that~$P_{\para_0}^{m_0}(x_0) = 1$.
  The next step is to recursively define a sequence~$(\para_i)_{i = 1}^{+\infty}$ in~$\pspace$, such that for every~$i$ in~$\Z_{\ge 0}$ the itinerary of~$x_0$ for~$P_{\para_i}$ is
  \begin{equation}
    \label{eq:80}
    \underbrace{0 \ldots 0}_{m_0}
    \underbrace{1 \ldots 1}_{M_0}
    \underbrace{0 \ldots 0}_{m_1}
    \ldots
    \underbrace{1 \ldots 1}_{M_{i -1}}
    \underbrace{0 \ldots 0}_{m_i}
    11 \ldots,
  \end{equation}
  and such that for all~$\para$ and~$\para'$ in~$\pspace$ satisfying ${|\para - \para_i| < \varepsilon_i}$ and ${|\para' - \para_i| < \varepsilon_i}$,
  \begin{equation}
    \label{eq:81}
    |P_\para^{n_i}(x) - P_{\para'}^{n_i}(x)|
    =
    |\para_0|^{-1} \cdot |\para - \para'|
  \end{equation}
  and, if ${i \ge 1}$, then ${|\para_i - \para_{i - 1}| = |a_0|^{-(M_{i - 1} - 1)}}$.

  To prove that~$\para_0$ satisfies the desired properties with ${i = 0}$, observe that the itinerary of~$x_0$ for~$P_{\para_0}$ is $\underbrace{0 \ldots 0}_{m_0} 11 \ldots$
  On the other hand, \cref{l:perturbation-1} with ${x' = x = x_0}$ and ${m = m_0}$ implies that for all~$\para$ and~$\para'$ in~$\pspace$ satisfying ${|\para - \para_0| < \varepsilon_0}$ and ${|\para' - \para_0| < \varepsilon_0}$,
  \begin{equation}
    \label{eq:82}
    |P_\para^{m_0}(x_0) - P_{\para'}^{m_0}(x_0)|
    =
    |\para_0|^{-1} \cdot |\para - \para'|.
  \end{equation}

  Let~$i$ be in~$\Z_{\ge 0}$ such that there is~$\para_i$ in~$\pspace$ with the desired properties.
  Then, the hypotheses of \cref{l:inductive-realization} are satisfied with ${\para_* = \para_i}$, ${n = n_i}$, ${\varepsilon = \varepsilon_i}$, ${M = M_i}$, and ${m = m_{i + 1}}$.
  The parameter provided by this lemma satisfies the desired properties with~$i$ replaced by~$i + 1$.
  This completes the proof that there is a sequence~$(\para_i)_{i = 1}^{+\infty}$ in~$\pspace$ with the desired properties.

  Since for every~$i$ in~$\Z_{> 0}$ the equality ${|a_i - a_{i - 1}| = |\para_0|^{-(M_{i - 1} - 1)}}$ holds, the sequence~$(\para_i)_{i = 1}^{+\infty}$ converges to a parameter~$\para$ in~$\pspace$.
  Evidently, $x_0$ has the desired itinerary for~$P_\para$.
\end{proof}

\section{Computing the diameter of a wandering domain}
\label{s:explicit-diameter}
This section proves the \cref{t:main} by explicitly computing the diameters of the iterates of a chosen ball as it passes through the wild ramification locus.

This computation requires the following notation.
Given a strictly increasing sequence of integers~$(\ell_s)_{s = 0}^{+\infty}$ satisfying~$\ell_0 = 0$, let~$(M_k)_{k = 0}^{+\infty}$ and~$(m_k)_{k = 0}^{+\infty}$ be defined in the statement of the \cref{t:main} and let~$(N_i)_{i = 0}^{+\infty}$ be defined recursively by~$N_0 \= 0$ and, for~$i$ in~$\Z_{> 0}$, by ${N_i \= N_{i - 1} + m_{i - 1} + M_{i - 1}}$.
Furthermore, for each~$s$ in~$\Z_{\ge 0}$ define
\begin{align}
  \label{eq:83}
  \delta_s
  & \=
    - \frac{m_s}{p - 1} + \frac{p}{(p - 1)^2} \left(1 - p^{-m_s} \right) + M_s,
  \\
  \label{eq:84}
  \tau_s
  & \=
    - \frac{p}{(p - 1)^2} \left( p^{-m_{\ell_s(p - 1)^2}} - p^{-m_{\ell_{s + 1}(p - 1)^2}} \right) - (\delta_{\ell_s(p - 1)^2} + \cdots + \delta_{\ell_{s + 1}(p - 1)^2 - 1}),
    \intertext{and, for each~$\para$ in~$\pspace$,}
    r_s(\para)
  &\=
    \label{eq:85}
    \rho_{m_{\ell_s(p - 1)^2}}(\para)^{\frac{p}{p - 1}}.
\end{align}
For every~$s$ in~$\Z_{\ge 0}$, the chain of inequalities ${0 < \tau_s < 1}$ holds (\cref{l:itinerary-recursivness}(3)), so ${0 < \sum_{u = 0}^{+\infty} \frac{\tau_{s + u}}{p^{u + 1}} < +\infty}$.

\begin{prop}
  \label{p:explicit-diameter}
  Let~$\para$ be a parameter in~$\pspace$ and let~$x_0$ be a point in~$\cK(P_\para)$ whose itinerary for~$P_\para$ is
  \begin{equation}
    \label{eq:86}
    \underbrace{0 \ldots 0}_{m_0}
    \underbrace{1 \ldots 1}_{M_0}
    \underbrace{0 \ldots 0}_{m_1}
    \underbrace{1 \ldots 1}_{M_1}
    \ldots
  \end{equation}
  Then, the ball~$B$ defined by
  \begin{equation}
    \label{eq:87}
    B
    \=
    \left\{ z \in K \: |z - x_0| \le r_0(\para) |\para|^{\sum_{u = 0}^{+\infty} \frac{\tau_u}{p^{u + 1}}} \right\},
  \end{equation}
  satisfies, for every~$s$ in~$\Z_{> 0}$,
  \begin{equation}
    \label{eq:88}
    \diam \left( P_\para^{N_{\ell_s(p - 1)^2}}(B) \right)
    =
    r_s(\para) |\para|^{\sum_{u = 0}^{+\infty} \frac{\tau_{s + u}}{p^{u + 1}}}.
  \end{equation}
\end{prop}

The proof of this proposition occupies \cref{ss:proof-explicit-diameter}, and that of the \cref{t:main} \cref{ss:proof-main-theorem}.
Recall from~\eqref{eq:2} that ${q \= (p - 1)(2p^2 - 2p - 1)}$.

\subsection{Proof of \cref{p:explicit-diameter}}
\label{ss:proof-explicit-diameter}
The proof relies on a couple of lemmas.

\begin{lemm}
  \label{l:itinerary-recursivness}
  For every~$s$ in~$\Z_{\ge 0}$, the following properties hold.
  \begin{enumerate}
  \item
    ${\delta_{\ell_{s + 1}(p - 1)^2 - 1}
      =
      -\frac{1}{(p - 1)^2} \left( 2p - 3 + p^{-(m_{\ell_{s + 1}(p - 1)^2 - 1} - 1)} \right) + (2p - 3)(\ell_{s + 1} - \ell_s).}$
  \item
    For every~$k$ in~$\Z_{\ge 0}$ satisfying ${\ell_s(p - 1)^2 - 1
      <
      k
      <
      \ell_{s + 1}(p - 1)^2 - 1}$,
    \begin{equation}
      \label{eq:89}
      m_k
      =
      (p - 1)M_k + 3,
      m_{k + 1}
      =
      m_k + 2(p - 1),
    \end{equation}
    and
    \begin{equation}
      \label{eq:90}
      \delta_k
      =
      -\frac{1}{(p - 1)^2} \left( 2p - 3 + p^{-(m_k - 1)} \right)
      <
      0.
    \end{equation}
  \item
    $0 < \tau_s < \frac{1}{(p - 1)^2 p^{2p}}$ and
    \begin{equation}
      \label{eq:91}
      \tau_s
      =
      \frac{1}{(p - 1)^2 p^{q\ell_s + 2p}} \left( \frac{1}{p^{2(p - 1)} - 1} + \frac{1}{p^{q(\ell_{s + 1} - \ell_s)}} - \frac{p^{2(p - 1)}}{p^{2(p - 1)} - 1} \cdot \frac{1}{p^{2(p - 1)^3(\ell_{s + 1} - \ell_s)}} \right).
    \end{equation}
  \end{enumerate}
\end{lemm}

\begin{proof}
  \hfill

  \partn{1}
  Put ${k \= \ell_{s + 1}(p - 1)^2 - 1}$.
  If ${k = 0}$, then ${p = 2}$, ${s = 0}$, ${\ell_1 = 1}$, ${M_0 = 3}$, ${m_0 = 5}$, ${\delta_0 = - \frac{1}{16}}$, and the desired formula for~$\delta_0$ follows from a straightforward computation.
  If ${k \ge 1}$, then
  \begin{equation}
    \label{eq:92}
    M_k
    =
    2k + 2 + (2p - 3)\ell_{s + 1},
    M_{k - 1}
    =
    2k + (2p - 3)\ell_s
    =
    M_k - 2 - (2p - 3)(\ell_{s + 1} - \ell_s),
  \end{equation}
  \begin{equation}
    \label{eq:93}
    m_k
    =
    (p - 1)M_{k - 1} + 2p + 1
    =
    (p - 1)M_k + 3 - (p - 1)(2p - 3)(\ell_{s + 1} - \ell_s),
  \end{equation}
  and
  \begin{multline}
    \label{eq:94}
    \delta_k
    =
    -\frac{(p - 1)M_k + 3}{p - 1} + (2p - 3)(\ell_{s + 1} - \ell_s) + \frac{1}{(p - 1)^2} \left( p - \frac{1}{p^{m_k - 1}} \right) + M_k
    \\ =
    -\frac{1}{(p - 1)^2} \left( 2p - 3 + \frac{1}{p^{m_k - 1}} \right) + (2p - 3)(\ell_{s + 1} - \ell_s).
  \end{multline}

  \partn{2}
  If ${k = 0}$, then ${s = 0}$ and the first equality in~\eqref{eq:89} follows from the definition of~$m_0$ and~$M_0$.
  If ${k \ge 1}$, then
  \begin{equation}
    \label{eq:95}
    M_k
    =
    2k + 2 + (2p - 3)\ell_s,
    M_{k - 1}
    =
    2k + (2p - 3)\ell_s
    =
    M_k - 2,
  \end{equation}
  and
  \begin{equation}
    \label{eq:96}
    m_k
    =
    (p - 1)M_{k - 1} + 2p + 1
    =
    (p - 1)M_k + 3.
  \end{equation}
  This proves the first equality in~\eqref{eq:89} in all of the cases.
  The second equality in~\eqref{eq:89} is a direct consequence of the first.
  Together with the definition of~$\delta_k$, the first equality in~\eqref{eq:89} implies
  \begin{multline}
    \label{eq:97}
    \delta_k
    =
    -\frac{(p - 1)M_k + 3}{p - 1} + \frac{1}{(p - 1)^2} \left( p - \frac{1}{p^{m_k - 1}} \right) + M_k
    \\ =
    -\frac{1}{(p - 1)^2} \left( 2p - 3 + \frac{1}{p^{m_k - 1}} \right)
    <
    0.
  \end{multline}

  \partn{3}
  For every~$s$ in~$\Z_{\ge 0}$,
  \begin{equation}
    \label{eq:98}
    m_{\ell_s(p - 1)^2}
    =
    q\ell_s + 2p + 1.
  \end{equation}
  Combined with items~$1$ and~$2$, this implies
  \begin{equation}
    \label{eq:99}
    \begin{split}
      \delta_{\ell_s(p - 1)^2} + \cdots + \delta_{\ell_{s + 1}(p - 1)^2 - 1}
      & =
        -\frac{p}{(p - 1)^2} \sum_{k = \ell_s(p - 1)^2}^{\ell_{s + 1}(p - 1)^2 - 1} p^{-m_k}
      \\ & =
           -\frac{p}{(p - 1)^2} \sum_{k = 0}^{(\ell_{s + 1} - \ell_s)(p - 1)^2 - 1} p^{-(q\ell_s + 2p + 1 + 2(p - 1)k)}
      \\ & =
           -\frac{p}{(p - 1)^2} \cdot \frac{p^{2(p - 1)}}{p^{2(p - 1)} - 1} \cdot \frac{1 - p^{-2(p - 1)^3(\ell_{s + 1} - \ell_s)}}{p^{q\ell_s + 2p + 1}}.
    \end{split}
  \end{equation}
  Using~\eqref{eq:98} again and the definition~\eqref{eq:84} of~$\tau_s$, this yields
  \begin{multline}
    \label{eq:100}
    \tau_s
    =
    \frac{p}{(p - 1)^2} \left( \frac{1}{p^{q\ell_s + 2p + 1}} \left( \frac{p^{2(p - 1)}}{p^{2(p - 1)} - 1} - 1 \right)
      + \frac{1}{p^{q\ell_{s + 1} + 2p + 1}} \right.
    \\ \left. -\frac{p^{2(p - 1)}}{p^{2(p - 1)} - 1} \cdot \frac{1}{p^{q\ell_s + 2p + 1 + 2(p - 1)^3(\ell_{s + 1} - \ell_s)}} \right)
  \end{multline}
  and thus~\eqref{eq:91}.
  Combined with the inequality ${\ell_{s + 1} - \ell_s \ge 1}$, this implies
  \begin{equation}
    \label{eq:101}
    0
    <
    \frac{1}{(p - 1)^2 p^{q\ell_{s + 1} + 2p}}
    \le
    \tau_s
    <
    \frac{1}{(p - 1)^2 p^{q\ell_s + 2p}} \left( \frac{1}{p^{2(p - 1)} - 1} + \frac{1}{p^q} \right)
    \le
    \frac{1}{(p - 1)^2 p^{2p}}.
    \qedhere
  \end{equation}
\end{proof}

\begin{lemm}
  \label{l:mildly-wild}
  For every parameter~$\para$ in~$\pspace$ and every~$s$ in~$\Z_{\ge 0}$, the following properties hold.
  \begin{enumerate}
  \item
    $ r_s(\para)
    <
    r_s(\para) |\para|^{\sum_{u = 0}^{+\infty} \frac{\tau_{s + u}}{p^{u + 1}}}
    <
    \rho(\para). $
  \item
    $ |\para| \left( r_s(\para) |\para|^{\sum_{u = 0}^{+\infty} \frac{\tau_{s + u}}{p^{u + 1}}} \right)^p
    <
    \rho(\para)^{\frac{p}{p - 1}}.$
  \item
    $r_s(\para) |\para|^{\delta_{\ell_s(p - 1)^2} + \sum_{u = 0}^{+\infty} \frac{\tau_{s + u}}{p^u}}
    <
    \rho(\para)^{\frac{p}{p - 1}}.$
  \end{enumerate}
\end{lemm}

\begin{proof}
  \hfill

  \partn{1}
  Using the definition of~$r_s(\para)$, the inequality ${m_{\ell_s(p - 1)^2} \ge 2p + 1}$, and \cref{l:itinerary-recursivness}(3),
  \begin{multline}
    \label{eq:102}
    r_s(\para)
    <
    r_s(\para) |\para|^{\sum_{u = 0}^{+\infty} \frac{\tau_{s + u}}{p^{u + 1}}}
    \le
    \rho_{2p + 1}(\para)^{\frac{p}{p - 1}} |\para|^{\sum_{u = 0}^{+\infty} \frac{\tau_{s + u}}{p^{u + 1}}}
    <
    \rho_{2p + 1}(\para)^{\frac{p}{p - 1}} |\para|^{\frac{1}{(p - 1)^3} \cdot \frac{1}{p^{2p}}}
    \\ =
    |\para|^{- \frac{p}{(p - 1)^2} + \frac{1}{(p - 1)^2} \cdot \frac{1}{p^{2p}} + \frac{1}{(p - 1)^3} \cdot \frac{1}{p^{2p}}}
    <
    |\para|^{- \frac{1}{p - 1}}
    =
    \rho(\para).
  \end{multline}

  \partn{2}
  Combining \cref{l:itinerary-recursivness}(3) and~$m_{\ell_s(p - 1)^2} \ge 2p + 1$, yields
  \begin{multline}
    \label{eq:103}
    \log_{|\para|} \left( |\para| \left( r_s(\para) |\para|^{\sum_{u = 0}^{+\infty} \frac{\tau_{s + u}}{p^{u + 1}}} \right)^p \right)
    =
    1 - \frac{p^2}{(p - 1)^2} \left( 1 - \frac{1}{p^{m_{\ell_s(p - 1)^2}}} \right) + \sum_{u = 1}^{+\infty} \frac{\tau_{s + u}}{p^u}
    \\ <
    -\frac{2p - 1}{(p - 1)^2} + \frac{1}{(p - 1)^2} \cdot \frac{1}{p^{2p - 1}} + \frac{1}{(p - 1)^3} \cdot \frac{1}{p^{2p - 1}}
    <
    -\frac{p}{(p - 1)^2}
    =
    \log_{|\para|} \left( \rho(\para)^{\frac{p}{p - 1}} \right).
  \end{multline}

  \partn{3}
  By \cref{l:itinerary-recursivness}(2) with ${k = \ell_s(p - 1)^2}$ and \cref{l:itinerary-recursivness}(3),
  \begin{multline}
    \label{eq:104}
    \log_{|\para|} \left( r_s(\para) |\para|^{\delta_{\ell_s(p - 1)^2} + \sum_{u = 0}^{+\infty} \frac{\tau_{s + u}}{p^u}} \right)
    =
    -\frac{3}{p - 1} + \sum_{u = 0}^{+\infty} \frac{\tau_{s + u}}{p^u}
    <
    -\frac{3}{p - 1} + \frac{1}{(p - 1)^3} \cdot \frac{1}{p^{2p - 1}}
    \\ <
    -\frac{p}{(p - 1)^2}
    =
    \log_{|\para|} \left(\rho(\para)^{\frac{p}{p - 1}} \right).
    \qedhere
  \end{multline}
\end{proof}

\begin{proof}[Proof of \cref{p:explicit-diameter}]
  For each~$s$ in~$\Z_{\ge 0}$, put
  \begin{equation}
    \label{eq:105}
    x_s
    \=
    P_\para^{N_{\ell_s(p - 1)^2}}(x_0)
    \text{ and }
    B_s
    \=
    P_\para^{N_{\ell_s(p - 1)^2}}(B).
  \end{equation}
  Then,
  \begin{equation}
    \label{eq:differential}
    \tau_s
    =
    \log_{|\para|}(r_{s + 1}(\para)/r_s(\para)) - (\delta_{\ell_s(p - 1)^2} + \cdots + \delta_{\ell_{s + 1}(p - 1)^2 - 1}).
  \end{equation}

  The proof proceeds by induction.
  Note that~\eqref{eq:88} holds with~$s = 0$ by the definition of~$N_0$ and~$B$.
  Let~$s$ in~$\Z_{\ge 0}$ be such that~\eqref{eq:88} holds.
  The next step is to prove by induction that, for every integer~$i$ satisfying
  \begin{equation}
    \label{eq:106}
    \ell_s(p - 1)^2 + 1
    \le
    i
    \le
    \ell_{s + 1}(p - 1)^2,
  \end{equation}
  \begin{equation}
    \label{eq:little-tour}
    \log_{|\para|}(\diam(P_\para^{N_i}(B))/r_s(\para))
    =
    \sum_{j = \ell_s(p - 1)^2}^{i - 1} \delta_j + \sum_{u = 0}^{+\infty} \frac{\tau_{s + u}}{p^u}.
  \end{equation}
  To prove this for ${i = \ell_s(p - 1)^2 + 1}$, note first that ${\left| x_s \right| = \rho_{m_{\ell_s(p - 1)^2}}}$ because the itinerary of~$x_s$ for~$P_\para$ starts with~$\underbrace{0 \ldots 0}_{m_{\ell_s(p - 1)^2}} 1$, see the paragraph just before \cref{l:contraction-expansion}.
  On the other hand,
  \begin{equation}
    \label{eq:107}
    \diam(B_s)
    =
    r_s(\para) |\para|^{\sum_{u = 0}^{+\infty} \frac{\tau_{s + u}}{p^{u + 1}}}
    <
    \rho(\para)
    <
    \rho_{m_{\ell_s(p - 1)^2}}(\para)
  \end{equation}
  by~\eqref{eq:88} and the second inequality in \cref{l:mildly-wild}(1).
  Together with \cref{l:contraction-expansion}(2) and the first inequality in \cref{l:mildly-wild}(1), this implies
  \begin{equation}
    \label{eq:108}
    \diam(P_\para(B_s))
    =
    |\para| \left( r_s(\para) |\para|^{\sum_{u = 0}^{+\infty} \frac{\tau_{s + u}}{p^{u + 1}}} \right)^p
    =
    \rho_{m_{\ell_s(p - 1)^2} - 1} (\para) r_s(\para) |\para|^{\sum_{u = 0}^{+\infty} \frac{\tau_{s + u}}{p^u}}.
  \end{equation}
  On the other hand, combining the first equality in~\eqref{eq:108} and \cref{l:mildly-wild}(2) yields
  \begin{equation}
    \label{eq:109}
    \diam(P_\para(B_s))
    <
    \rho(\para)^{\frac{p}{p - 1}}
    <
    \rho_{m_{\ell_s(p - 1)^2} - 1}(\para)^{\frac{p}{p - 1}}.
  \end{equation}
  Thus, applying \cref{l:contraction-expansion}(1) repeatedly yields
  \begin{equation}
    \label{eq:110}
    \diam \left( P_\para^{m_{\ell_s(p - 1)^2}}(B_s) \right)
    =
    \rho_0 (\para) \ldots \rho_{m_{\ell_s(p - 1)^2} - 1} (\para) r_s(\para) |\para|^{\sum_{u = 0}^{+\infty} \frac{\tau_{s + u}}{p^u}}
  \end{equation}
  and thus
  \begin{equation}
    \label{eq:111}
    \log_{|\para|} \left( \diam \left( P_\para^{m_{\ell_s(p - 1)^2}}(B_s) \right) / r_s(\para) \right)
    =
    \delta_{\ell_s(p - 1)^2} - M_{\ell_s (p - 1)^2} + \sum_{u = 0}^{+\infty} \frac{\tau_{s + u}}{p^u}.
  \end{equation}
  Together with \cref{l:mildly-wild}(3), this implies
  \begin{equation}
    \label{eq:112}
    \diam(P_\para^{m_{\ell_s(p - 1)^2}}(B_s))
    <
    |\para|^{- M_{\ell_s (p - 1)^2}}.
  \end{equation}
  Thus, applying \cref{l:contraction-expansion}(3) repeatedly yields
  \begin{multline}
    \label{eq:113}
    \log_{|\para|} \left( \diam \left( P_\para^{N_{\ell_s(p - 1)^2 + 1}}(B) \right) / r_s(\para) \right)
    =
    \log_{|\para|} \left( \diam \left( P_\para^{m_{\ell_s(p - 1)^2} + M_{\ell_s(p - 1)^2}}(B_s) \right) / r_s(\para) \right)
    \\ =
    \delta_{\ell_s(p - 1)^2} + \sum_{u = 0}^{+\infty} \frac{\tau_{s + u}}{p^u}.
  \end{multline}
  This proves~\eqref{eq:little-tour} with ${i = \ell_s(p - 1)^2 + 1}$.
  Let~$i$ be an integer satisfying
  \begin{equation}
    \label{eq:114}
    \ell_s(p - 1)^2 + 1
    \le
    i
    \le
    \ell_{s + 1}(p - 1)^2 - 1
  \end{equation}
  and~\eqref{eq:little-tour}.
  Since the itinerary of~$P_\para^{N_i}(x_0)$ for~$P_\para$ starts with~$\underbrace{0 \ldots 0}_{m_i} 1$, the equality ${\left| P_\para^{N_i}(x_0) \right| = \rho_{m_i}}$ holds, see the paragraph just before \cref{l:contraction-expansion}.
  On the other hand, ${\delta_{\ell_s(p - 1)^2 + 1} + \cdots + \delta_i < 0}$ by \cref{l:itinerary-recursivness}(2), and hence~\eqref{eq:little-tour} and \cref{l:mildly-wild}(3) yield
  \begin{equation}
    \label{eq:115}
    \diam \left( P_\para^{N_i}(B) \right)
    \le
    r_s(\para) |\para|^{\delta_{\ell_s(p - 1)^2} + \sum_{u = 0}^{+\infty} \frac{\tau_{s + u}}{p^u}}
    <
    \rho(\para)^{\frac{p}{p - 1}}
    <
    \rho_{m_{\ell_s(p - 1)^2 + i}}(\para)^{\frac{p}{p - 1}}.
  \end{equation}
  Applying \cref{l:contraction-expansion}(1) repeatedly yields
  \begin{multline}
    \label{eq:116}
    \diam \left( P_\para^{N_i + m_i}(B) \right)
    =
    \rho_0 (\para) \ldots \rho_{m_i - 1} (\para) r_s(\para) |\para|^{\delta_{\ell_s(p - 1)^2} + \cdots + \delta_{i - 1} + \sum_{u = 0}^{+\infty} \frac{\tau_{s + u}}{p^u}}
    \\ =
    r_s(\para) |\para|^{\delta_{\ell_s(p - 1)^2} + \cdots + \delta_i - M_i + \sum_{u = 0}^{+\infty} \frac{\tau_{s + u}}{p^u}}.
  \end{multline}
  If ${i < \ell_{s + 1}(p - 1)^2 - 1}$, then ${\delta_{\ell_s(p - 1)^2 + 1} + \cdots + \delta_i < 0}$ by \cref{l:itinerary-recursivness}(2), and the last term in~\eqref{eq:116} is less than~$|\para|^{-M_i}$ by \cref{l:mildly-wild}(3).
  If ${i = \ell_{s + 1}(p - 1)^2 - 1}$, then~\eqref{eq:differential} implies that the last term in~\eqref{eq:116} is equal to
  \begin{equation}
    \label{eq:117}
    r_{s + 1} (\para) |\para|^{-M_{\ell_{s + 1}(p - 1)^2} + \sum_{u = 1}^{+\infty} \frac{\tau_{s + u}}{p^u}}.
  \end{equation}
  Together with the second inequality in \cref{l:mildly-wild}(1) with~$s$ replaced by~$s + 1$, this implies that the last term in~\eqref{eq:116} is less than~$|\para|^{-M_i}$.
  In all of the cases, applying \cref{l:contraction-expansion}(3) repeatedly yields
  \begin{multline}
    \label{eq:118}
    \log_{|\para|} \left( \diam \left( P_\para^{N_{i + 1}}(B) \right) / r_s(\para) \right)
    =
    \log_{|\para|} \left( \diam \left( P_\para^{N_i + m_i + M_i}(B) \right) / r_s(\para) \right)
    \\ =
    \delta_{\ell_s(p - 1)^2} + \cdots + \delta_i + \sum_{u = 0}^{+\infty} \frac{\tau_{s + u}}{p^u}.
  \end{multline}
  This proves~\eqref{eq:little-tour} with~$i$ replaced by~$i + 1$ and completes the proof that~\eqref{eq:little-tour} holds for every integer~$i$ satisfying~\eqref{eq:106}.

  Combining~\eqref{eq:differential} and~\eqref{eq:little-tour} with ${i = \ell_{s + 1}(p - 1)^2}$, yields
  \begin{equation}
    \label{eq:119}
    \log_{|\para|} \diam(B_{s + 1})
    =
    \log_{|\para|} r_s(\para) + \sum_{j = \ell_s(p - 1)^2}^{\ell_{s + 1}(p - 1)^2 - 1} \delta_j + \sum_{u = 0}^{+\infty} \frac{\tau_{s + u}}{p^u}
    =
    \log_{|\para|} r_{s + 1} (\para) + \sum_{u = 1}^{+\infty} \frac{\tau_{s + u}}{p^u}.
  \end{equation}
  This completes the proof that~\eqref{eq:88} holds with~$s$ replaced by~$s + 1$, and shows that~\eqref{eq:88} holds for every~$s$ in~$\Z_{\ge 0}$, as wanted.
\end{proof}

\subsection{Proof of the \cref{t:main}}
\label{ss:proof-main-theorem}
Note that ${M_0 \ge 2}$ by definition, and that, for each~$i$ in~$\Z_{\ge 0}$, the inequality ${M_{i + 1} \ge M_i + 2}$ holds and
\begin{equation}
  \label{eq:120}
  M_i - \frac{m_{i + 1}}{p - 1} + \frac{p}{(p - 1)^2} \left(1 - \frac{1}{p^{m_{i + 1}}} \right)
  =
  - \frac{1}{(p - 1)^2} \left(2p^2 - 2p - 1 + \frac{1}{p^{m_{i + 1} - 1}} \right)
  <
  0.
\end{equation}
So, the sequences~$(m_i)_{i = 0}^{+\infty}$ and~$(M_i)^{+\infty}$ satisfy the hypotheses of \cref{p:itinerary-realization}.
Thus, there is~$\para$ in~$\pspace$ satisfying ${|\para - \para_0| = |\para_0|^{-(M_0 - 1)}}$ and~$x_0$ in~$\cK(P_\para)$ whose itinerary for~$P_\para$ is
\begin{equation}
  \label{eq:121}
  \underbrace{0 \ldots 0}_{m_0}
  \underbrace{1 \ldots 1}_{M_0}
  \underbrace{0 \ldots 0}_{m_1}
  \underbrace{1 \ldots 1}_{M_1}
  \ldots
\end{equation}

For each~$s$ in~$\Z_{\ge 0}$, put
\begin{equation}
  \label{eq:122}
  t_s
  \=
  \log_{|\para|} r_s(\para) + \sum_{u = 0}^{+\infty} \frac{\tau_{s + u}}{p^{u + 1}}
  \text{ and }
  n_s
  \=
  N_{m_{\ell_s(p - 1)^2}}.
\end{equation}
\cref{l:mildly-wild}(1) yields
\begin{equation}
  \label{eq:123}
  \rho(\para)^{\frac{p}{p - 1}}
  <
  \rho_{m_{\ell_s(p - 1)^2}}(\para)^{\frac{p}{p - 1}}
  <
  |\para|^{t_s}
  <
  \rho(\para).
\end{equation}
Moreover, the ball~$B$ defined in the statement of \cref{p:explicit-diameter} is equal to ${\{z \in K \: |z - x_0| \le |\para|^{t_0}\}}$.
Thus, for every~$s$ in~$\Z_{\ge 0}$,
\begin{equation}
  \label{eq:124}
  \diam \left( P_\para^{n_s}(B) \right)
  =
  |\para|^{t_s}
  <
  \rho(\para)
\end{equation}
and hence ${B \subseteq \cK(P_\para)}$.

The next step is to prove ${t_0 = t}$, so~$B$ coincides with the unnamed ball defined in the statement of the \cref{t:main}.
By \eqref{eq:2} and \cref{l:itinerary-recursivness}(3),
\begin{multline}
  \label{eq:125}
  (p - 1)^2 p^{2p + 1}(p^{2(p - 1)} - 1) \sum_{u = 0}^{+\infty} \frac{\tau_u}{p^{u + 1}}
  \\
  \begin{aligned}
    & =
      \sum_{u = 0}^{+\infty} \left( \frac{1}{p^{q\ell_u + u}} + \frac{p^{2(p - 1)} - 1}{p^{q\ell_{u + 1} + u}} - \frac{1}{p^{q\ell_u + u}} \cdot \frac{p^{2(p - 1)}}{p^{2(p - 1)^3(\ell_{u + 1} - \ell_u)}} \right)
    \\ & =
         -p(p^{2(p - 1)} - 1) + (p^{2p - 1} - p + 1) \sum_{u = 0}^{+\infty} \frac{1}{p^{q\ell_u + u}}
         -\sum_{u = 0}^{+\infty} \frac{1}{p^{q\ell_u + u}} \cdot \frac{p^{2(p - 1)}}{p^{2(p - 1)^3(\ell_{u + 1} - \ell_u)}}
    \\ & =
         -p(p^{2(p - 1)} - 1) + (p^{2p - 1} - p + 1) \sum_{u = 0}^{+\infty} \frac{1}{p^{q\ell_u + u}} \left(1 - \frac{\kappa}{p^{2(p - 1)^3(\ell_{u + 1} - \ell_u)}} \right).
  \end{aligned}
\end{multline}
Thus,
\begin{equation}
  \label{eq:126}
  t_0
  =
  \log_{|\para|} r_0(\para) + \sum_{u = 0}^{+\infty} \frac{\tau_u}{p^{u + 1}}
  =
  -\frac{p}{(p - 1)^2} \left( 1 - \frac{1}{p^{2p + 1}} \right) + \sum_{u = 0}^{+\infty} \frac{\tau_u}{p^{u + 1}}
  =
  t.
\end{equation}

Since~$M_i \to +\infty$ as~$i \to +\infty$, the itinerary~\eqref{eq:121} is not preperiodic.
So, by \cref{l:preperiodic-or-wandering} it only remains to prove that~$B$ is a component of~$\cK(P_\para)$.
That is, that for every~$t'$ in~$(t_0, +\infty)$ the disk~$D'$, defined by
\begin{equation}
  \label{eq:127}
  D'
  \=
  \{ z \in K \: |z - x| < |\para|^{t'} \},
\end{equation}
is not contained in~$\cK(P_\para)$.
Suppose there were such a~$t'$ so that ${D' \subseteq \cK(P_\para)}$, and for each~$s$ in~$\Z_{\ge 0}$ define
\begin{equation}
  \label{eq:128}
  t_s'
  \=
  \log_{|\para|} \diam \left( P_\para^{n_s}(D') \right).
\end{equation}
Combining the second inequality in~\eqref{eq:123} with \cref{l:contraction-expansion}(2), yields
\begin{equation}
  \label{eq:129}
  \diam \left( P_\para^{n_s + 1}(D') \right) / \diam \left( P_\para^{n_s + 1}(B) \right)
  =
  |\para|^{p |t_s' - t_s|}.
\end{equation}
Thus, by \textsc{Schwarz}' Lemma in~\cite[\S1.3.1]{Riv03c} applied to~$P_\para^{n_{s + 1} - n_s - 1}$,
\begin{multline}
  \label{eq:130}
  |\para|^{|t_{s + 1}' - t_{s + 1}|}
  =
  \diam \left( P_\para^{n_{s + 1}}(D') \right) / \diam \left( P_\para^{n_{s + 1}}(B) \right)
  \\ \ge
  \diam \left( P_\para^{n_s + 1}(D') \right) / \diam \left( P_\para^{n_s + 1}(B) \right)
  =
  |\para|^{p |t_s' - t_s|}
\end{multline}
and hence ${|t_{s + 1}' - t_{s + 1}| \ge p |t_s' - t_s|}$.
Together with~\eqref{eq:123} and an induction argument, this implies ${t_s' \to +\infty}$ as ${s \to +\infty}$.
This contradicts the hypothesis ${D' \subseteq \cK(P_\para)}$ and completes the proof of the \cref{t:main}.

\bibliographystyle{alpha}

\end{document}